\documentclass[a4paper,10pt]{amsart}
\usepackage{amsmath,amssymb,enumerate,graphicx,url,amscd, color}
\usepackage{amsthm}

\usepackage[utf8]{inputenc}

\def\del{\partial}
\def\Ci{C^\infty}
\def\dv{d_{\vertical}}
\def\hCi{\widehat{C}^\infty}
\def\Li{L_\infty}
\def\nr{\mathbb R}
\def\nz{\mathbb Z}
\DeclareMathOperator{\Hom}{Hom}

\DeclareMathOperator{\vertical}{vert}
\providecommand{\abs}[1]{\lvert#1\rvert}

\newtheorem{theorem}{Theorem}[section]
\newtheorem{lemma}[theorem]{Lemma}
\newtheorem{proposition}[theorem]{Proposition}
\newtheorem{corollary}[theorem]{Corollary}

\theoremstyle{definition}
\newtheorem{definition}[theorem]{Definition}
\newtheorem{example}[theorem]{Example}
\newtheorem{examples}[theorem]{Examples}

\theoremstyle{remark}
\newtheorem*{remark}{Remark}

\newcommand{\tocless}[2]{\bgroup\let\addcontentsline=\nocontentsline#1{#2}
\egroup}

\title{On homotopy Lie bialgebroids}

\author[D. Bashkirov]{Denis Bashkirov}
\email{bashk003@umn.edu}
\address {School of Mathematics\\University of Minnesota\\
  Minneapolis, MN 55455, USA}

\author[A. A. Voronov]{Alexander A. Voronov} \email{voronov@umn.edu}
\address {School of Mathematics\\University of Minnesota\\
  Minneapolis, MN 55455, USA, and Kavli IPMU (WPI), UTIAS, University
  of Tokyo, Kashiwa, Chiba 277-8583, Japan}

\date{August 16, 2017}

\begin{document}

\begin{abstract}
  Lie algebroids appear throughout geometry and mathematical physics
  implementing the idea of a sheaf of Lie algebras acting
  infinitesimally on a smooth manifold.  A well-known result of
  A.~Vaintrob characterizes Lie algebroids and their morphisms in
  terms of homological vector fields on supermanifolds, which might be
  regarded as objects of derived geometry. This leads naturally to the
  notions of an $\Li$-algebroid and an $\Li$-morphism. The situation
  with Lie bialgebroids and their morphisms is more complicated, as
  they combine covariant and contravariant features. We approach Lie
  bialgebroids in the way of odd symplectic dg-manifolds, building on
  D.~Roytenberg's thesis. We extend Lie bialgebroids to the homotopy
  Lie case and introduce the notions of an $\Li$-bialgebroid and an
  $\Li$-morphism. The case of $\Li$-bialgebroids over a point
  coincides with the O.~Kravchenko's notion of an $\Li$-bialgebra, for
  which the notion of an $\Li$-morphism seems to be new.
\end{abstract}

\maketitle

\tableofcontents

\section{Introduction}

The notion of a Lie bialgebra was introduced in the seminal works of
V.~Drinfeld \cite{Dr1,Dr2} on algebraic aspects of the quantum inverse
scattering method. A Lie bialgebra $\mathfrak{g}$ is a Lie algebra
equipped with a one-cocycle $\delta:\mathfrak{g}\to \mathfrak{g}\wedge
\mathfrak{g}$ (a \textit{cobracket}), whose dual $\delta^*$ yields a
Lie bracket on $\mathfrak{g}$.  As a quintessential example, Lie
bialgebras appear as infinitesimal counterparts of Poisson-Lie
groups. Geometrization of this notion leads to the concept of a Lie
bialgebroid that natually arises in the Poisson-geometric context. In
particular, there is a canonical Lie bialgebroid associated to any
Poisson manifold. The aim of this note is to introduce an extension of
this concept to the case of graded manifolds and homotopy Lie
structures.

We survey the basic definitions, motivating examples and results
concerning Lie algebroids, $\Li$-algebroids, and Lie bialgebroids in
sections 2 and 3.  In section 4 we review the Hamiltonian approach to
Lie (bi)algebroids of D.~Roytenberg and then give a Hamiltonian
characterization of Lie (bi)algebroid morphisms
(Theorem~\ref{Lie-bialg-mor}). In the final section, we introduce the
notions of an $L_\infty$-bialgebroid and an $L_\infty$-morphism of
$L_\infty$-bialgebroids and list some relevant examples.

\subsection*{Conventions}

The ground field is $\nr$ by default. The dual $V^*$ of a graded
vector bundle $V$ is understood as the direct sum of the duals of its
graded components, graded in such a way that the natural pairing $V^*
\otimes V \to \nr$ is grading-preserving.  In particular, $(V[n])^* =
V^*[-n]$. By default, the degree on a bigraded vector bundle, such as
$S(V)$, stands for the total degree.  Differentials are assumed to
have degree 1.

We use the exterior algebra $\wedge^\bullet V$ and the symmetric
algebra $S(V[-1])$ interchangeably. The former is used mostly for
ungraded vector spaces $V$, whereas the latter is reserved for graded
ones.

We assume vector bundles to have finite rank and graded vector bundles
to have locally finite rank. Likewise, all graded manifolds will be
assumed to have finited-dimensional graded components. We will work
with smooth \emph{graded manifolds}, which we will understand as
locally ringed spaces $(M, C^\infty_{V}) := (M, S(\mathcal{V}^*))$,
where $V \to M$ is a graded vector bundle over a manifold $M$ and
$S(\mathcal{V}^*)$ is the graded symmetric algebra over $C^\infty_M$
on the graded dual to the sheaf $\mathcal{V}$ of sections.

A \emph{morphism $V \to W$ of graded manifolds} is a morphism of
locally ringed spaces $(M, S(\mathcal{V}^*)) \to (N,
S(\mathcal{W}^*))$. will assume that a \emph{differential},
\emph{i.e}., a degree-one, $\nr$-linear derivation $d$ of the
structure sheaf satisfying $d^2 = 0$ is given on a graded
manifold. \emph{Morphisms of dg-manifolds} will have to respect
differentials. Since we work in the $\Ci$ category, we will routinely
substitute sheaves with spaces of their global sections.

A graded manifold $V$ over $M$ comes with a morphism $V \to M$ given
by the inclusion of $\Ci_M$ into the function sheaf of $V$. There is
also a \emph{relative basepoint} given by the zero section of the
vector bundle $V \to M$, which induces a morphism $M \to V$ of graded
manifolds over $M$. A \emph{based morphism} must respect zero sections
of the structure vector bundles. The structure differential will also
be assumed to be \emph{based}, that is to say, the zero section $M \to
V$ must be a dg-morphism.

We will follow a common trend and confuse the notation $V$ for a
vector bundle $V$ and the sheaf $\mathcal{V}$ of its sections, when it
is clear what we mean from the context.

% The translation $V[n]$ of a graded vector bundle $V$ is the same vector
% bundle with a redefined degree: $V[n]^p := V^{p+n}$.
 
For a manifold $M$ with a Poisson tensor $\pi$, we will use $\pi^\#$
to denote the natural morphism $T^*M\to TM$ determined by the
condition
\[
 \langle \alpha\wedge\beta, \pi \rangle = \langle \beta, 
\pi^\#(\alpha)\rangle,\quad \alpha,\beta\in T^*M.
\]

\section{Lie algebroids}

\subsection{Basic definitions and examples}
\begin{definition}
 A \textit{Lie algebroid} structure on a vector bundle $V\to M$ over a 
smooth manifold $M$ consists of
 \begin{itemize}
  \item 
  an $\nr$-bilinear Lie bracket $[,]:\Gamma(V)\otimes_\nr \Gamma(V)\to
  \Gamma(V)$ on the space of sections;
  \item
  a morphism of vector bundles $\rho: V\to TM$, called the \textit{anchor map},
 \end{itemize}
subject to the Leibniz rule
  $$[X,fY]=f[X,Y]+(\rho(X)(f))Y, \quad X,Y\in \Gamma(V),\,f\in
 C^{\infty}(M).$$
It follows, in particular, that the anchor map is a morphism of Lie
algebras:
\begin{equation}
\label{anchor-Lie}
   \rho([X,Y])=[\rho(X),\rho(Y)], \quad X,Y\in \Gamma(V).
\end{equation}

\end{definition}
\begin{examples}
\label{basic-examples}
\noindent
\newline
 \begin{enumerate}
  \item 
Any Lie algebra can be regarded as a Lie algebroid over a point.
  \item
\label{tangent}
The tangent bundle $TM$ taken with the standard Lie bracket of vector
fields and $\rho=id:TM\to TM$ is trivially a Lie algebroid.
  \item
\label{2}
More generally, any integrable distribution $V\subset TM$ is a Lie
algebroid with $\rho: V\to TM$ being the inclusion. Thus, a regular foliation
on a manifold gives rise to a Lie algebroid.
\item A family of Lie algebras over a manifold $M$, \emph{i.e}., a
  vector bundle with a Lie bracket bilinear over functions on $M$, is
  a Lie algebroid with a zero anchor.
 \item
 \label{3}
 Let $\mathfrak{g}$ be a Lie algebra acting on a manifold $M$ via an
 infinitesimal action map $\mathfrak{g}\to \Gamma(TM)$, $X \mapsto
 \xi_X$. Then $\mathfrak{g}\times M\to M$ is a Lie algebroid with
 $(X,m)\mapsto \xi_X (m)$ as the anchor and the bracket defined
 pointwise by
\[
 [X,Y](m) := [X(m),Y(m)]_{\mathfrak{g}} + \xi_X Y (m)-\xi_YX (m),
\]
where a section $X$ of $\mathfrak{g} \times M \to M$ is identified
with a function $X: M \to \mathfrak{g}$.

\item
Every Lie groupoid gives rise to a Lie algebroid, see Example
\ref{bi-examples}\eqref{groupoid}. For example the tangent Lie
algebroid of Example~\eqref{tangent} above comes from the \emph{pair
  groupoid} of a manifold $M$. The manifold of objects of the pair
groupoid is $M$ itself, the manifold of morphisms is $M \times M$,
with morphism composition given by
\[
(x,y) \circ (y,z) := (x,z), \qquad x, y, z \in M.
\]

 \item
\label{5}
If $M$ is a Poisson manifold with the Poisson bivector $\pi\in
\Gamma(\wedge^2 TM)$, then the canonical morphism $\pi^\#:T^*M\to TM$
together with the Koszul bracket
\[
 \{\alpha,\beta\}_\pi= L_{\pi^\#(\alpha)}(\beta)- L_{\pi^\#(\beta)}
(\alpha)-d(\iota_\pi(\alpha\wedge \beta))
\]
determines a Lie algebroid structure on $T^*M$.
 \item
\label{6}
Given a vector bundle $V\to M$, the space of \textit{derivative
  endomorphisms} $Der(V)$ is defined as the space of all linear
endormorphisms $D:\Gamma(V)\to \Gamma(V)$ such that there exists
$D_M\in TM$, and $$D(fX)=fD(X)+D_M(f)X$$ for any $X\in \Gamma(V)$,
$f\in C^\infty(M)$.  Then $Der(V)$ equipped with the standard
commutator bracket and the mapping $\rho: D\mapsto D_M$ as the anchor
is a Lie algebroid.
 
\item For a principal $G$-bundle $P$ over a manifold $M$, the quotient
  $TP/G$ of $TP$ by the induced action of $G$ is known as the
  \textit{Atiyah Lie algebroid} of $P$. The bracket and the anchor map
  are naturally inherited from $TP$.
 \end{enumerate}
 \end{examples}
The notion of a morphism of Lie algebroids $V\to W$ defined over the
same base manifold $M$ is rather straightforward: it is a vector
bundle morphism $\phi: V\to W$ such that
$\phi([X,Y])=[\phi(X),\phi(Y)]$ and $\rho_W\circ \phi=\rho_V$.  In
general, the definition of such a morphism in terms of brackets and
anchor maps is more involved due to the fact that a morphism of vector
bundles defined over different bases does not induce a morphism of
sections.
 
To introduce relevant notation, let $\phi:V\to W$ be a morphism of
vector bundles $V\to M$, $W\to N$ over $f:M\to N$ and $\phi_!:V\to
f^*W$ be a canonical morphism arising from the universal property of
the pullback $f^*W$. This induces a mapping of sections
$$\Gamma(V)\to \Gamma(f^*W)\simeq
C^\infty(M)\otimes_{C^\infty(N)}\Gamma(W)$$ that we, by a slight abuse
of notation, will keep denoting by $\phi_!$.
\begin{definition}
\label{Lie-morphism}
Then a \textit{morphism $V \to W$ of Lie algebroids} is a morphism
\[
\begin{CD}
V @>\phi>> W\\
@VVV @VVV\\
M @>>> N
\end{CD}
\]
of vector bundles subject to the following conditions:
  \begin{itemize}
   \item 
   $\rho_W\circ \phi=df\circ \rho_V$;
 \item for $X,Y\in \Gamma(V)$, write
   $\phi_!(X)=\sum\limits_{i}f_i\otimes X'_i$,
   $\phi_!(Y)=\sum\limits_{j}g_j\otimes Y'_j$; then
   \[
    \phi_!([X,Y])=\sum\limits_{i,j}f_i g_i \otimes
        [X'_i,Y'_j]+(\rho_V(X)(g_j))\otimes
        Y'_j-(\rho_V(Y)(f_i))\otimes X'_i.
   \]
  \end{itemize}
\end{definition} 

It is an exercise to check that the second condition is independent of
the tensor-product expansions.

 \begin{examples}
 \noindent
 \newline
 \begin{enumerate}
 \item
For any smooth map $f:M\to N$, the tangent map $df: TM\to TN$ is a
morphism of tangent Lie algebroids as defined in
Example \ref{basic-examples}\eqref{2}.
\item
Given a Lie algebroid $V \to M$, its anchor map $TM \to V$ is a
morphism of Lie algebroids.
\item
In Example \ref{basic-examples}\eqref{6} above, Lie algebroid
morphisms $TM\to Der(V)$ right-inverse to the anchor map
$\rho:Der(V)\to TM$ correspond to flat connection on $V$.
 \end{enumerate}
 \end{examples}
 
 \subsection{The dg-manifold approach}

 Given a Lie algebroid $V\to M$, the coboundary operator 
$d:\Gamma(\wedge^{k} V^*)\to\Gamma(\wedge^{k+1} V^*)$
 defined by 
 \begin{align*}
   d\varphi(X_1,\dots,X_{k+1})&=\sum\limits_{i=1}^{k+1}(-1)^{i+1}\rho(X_i)\varphi(X_1,\dots,\hat{X_i},\dots,X_{k+1})
   \\  
                              &+\sum\limits_{i<j}(-1)^{i+j}\varphi([X_i,X_j],X_1,\dots,\hat{X_i},\dots,\hat{X_j},\dots,X_{k+1})\tag{1}
 \end{align*}
 turns $\Gamma(\wedge^\bullet V^*)$ into a differential graded algebra.
 \begin{examples}
 \noindent
 \newline
 \begin{enumerate}
 \item For a Lie algebra, this yields the standard cohomological
   Che\-val\-ley-Ei\-len\-berg complex with trivial coefficients.
 \item If $V=TM$ is a tangent Lie algebroid, then
   $\Omega^\bullet (M) := \Gamma(\wedge^\bullet T^*M) $ is the
   standard de Rham complex of a smooth manifold $M$.
\item For the Lie algebroid associated with a Poisson manifold $M$ as
  in Example \ref{basic-examples}\eqref{5}, $\Gamma(\wedge^\bullet
  TM)$ is the cohomological Poisson complex. The differential
  $d_\pi=[\pi,-]$ in this case is known as the \textit{Lichnerowicz}
  differential. The standard mapping $\pi^\#:T^*M\to TM$ induces a
  morphism between the de Rham cohomology $H^\bullet_{dR}(M)$ and the
  cohomology of $(\Gamma(\wedge^{\bullet} TM),d_{\pi})$, which turns
  out to be an isomorphism in the symplectic case.
 \end{enumerate}
 \end{examples}
 \begin{remark}
  In fact, $\Gamma(\wedge^\bullet V^*)$ admits a slightly richer structure.
  Namely, the contraction $i_X:\Gamma(\wedge^{k+1} V^*)\to\Gamma(\wedge^{k} V^*)$ and the Lie derivative $L_X$ defined 
  by
  \[
   (L_X\varphi)(Y_1,\dots,Y_k)=\rho(X)(\varphi(Y_1,\dots,Y_k))-\sum\limits_{i=1}^{k}\varphi(Y_1,\dots,[X,Y_i],\dots,Y_k)
  \]
  satisfy all the standard rules of Cartan calculus.
 \end{remark}

Passage from Lie algebroids to the ``Koszul dual'' picture encoded by
the corresponding dg-algebras simplifies the matters concerning Lie
algebroid morphisms. This is due to the following
 \begin{theorem}[A.~Vaintrob \cite{vaintrob}]
 \label{Vaintrob}
   Let $V\to M$ be a vector bundle. Then the structures of
\begin{enumerate}
\item a Lie algebroid on $V$,
\item a dg-manifold, $d^V: C^\infty (V[1]) \to C^\infty(V[1])$, on the
  graded manifold $V[1]$
\end{enumerate}
are equivalent.
%\footnote{The structure of a dg-algebra on a graded
%  algebra boils down to a differential compatible with the
%  multiplication.} 
Furthermore, there are natural bijections between
the following sets:
\begin{enumerate}
\item The set of morphisms of Lie algebroids $V\to M$ and $W\to N$;
\item The set of dg-manifold morphisms $(V[1], d^V) \to (W[1], d^W)$.
\end{enumerate}
 \end{theorem}

%\begin{remark}
%***Homological vector fields***
In the context of graded manifolds the differential $d^V$
%, which is a part of a dg-manifold structure on 
%$V[1]$, 
is commonly referred to as a \textit{homological vector field} on
$V[1]$. The complex $S(V^*[-1]), d^V)$ is often called the
\emph{cohomological Chevalley-Eilenberg complex} of the Lie algebroid.
%\end{remark}

\begin{proof}[Idea of proof]
  A derivation $d^V: C^\infty(V[1]) \to C^\infty(V[1])$ of degree one
  is determined by its restriction to the subalgebra $C^\infty(M)$:
\begin{equation}
\label{derivation}
C^\infty(M) \to \Gamma(V^*[-1]),
\end{equation}
which must be a derivation of the algebra $C^\infty(M)$ with values in
a $C^\infty(M)$-module, and by the restriction to the module of
generators:
\begin{equation}
\label{dual-bracket}
 \Gamma(V^*[-1]) \to  \Gamma(S^2(V^*[-1])).
\end{equation}
By the universal property of K\"ahler differentials,
\eqref{derivation} is equivalent to a $C^\infty(M)$-module morphism
\[
\Omega^1(M) \to \Gamma(V^*[-1]),
\]
whose dualization gives an anchor. The $\nr$-dual of
\eqref{dual-bracket} gives a bracket. The differential property
$(d^V)^2 = 0$ then translates into the Jacobi identity for the bracket
and the Lie algebra morphism property for the anchor. The Lebniz rule
for $d^V$ translates into the Leibniz rule for the bracket.
\end{proof}

\begin{example}
\label{derham}
The tangent algebroid $TM \to M$ of a (graded) manifold $M$
corresponds to the graded manifold $T[1]M$, a shifted tangent bundle,
whose dg-algebra of smooth functions is the de Rham algebra
$(\Omega^\bullet (M), d_{dR})$.
\end{example}

\subsection{The graded case and $\Li$-algebroids}

Treating the notion of a Lie algebroid via dg-manifolds leads
naturally to its graded version known as an $\Li$-algebroid.  The
concept of an $\Li$-algebroid was conceived in the works of
H.~Khudaverdian and Th.~Th. Voronov \cite{ThVor2}, H.~Sati,
U.~Schreiber and J.~Stasheff \cite{sati-schreiber-stasheff2009},
A.~J. Bruce \cite{bruce2010}, and Th.~Th. Voronov \cite{fedya2010}.
At the same time, $\Li$-algebras disguised as formal dg-manifolds or
%in a closely related way,% as supermanifolds with an
$\nr^{0|1}$-action have been known to V.~Drinfeld, V.~A. Hinich,
M.~Kontsevich, D.~Quillen, V.~Schechtman, D.~Sullivan, and likely some
others, since the last three decades of the 20th century; see
\cite{kontsevich:feynman} and \cite{stasheff:1609.08401} and
references therein. $\Li$-algebroids have made their way to physics,
as one of the most general models of quantum field theory, the AKSZ
model \cite{AKSZ}; it is a sigma model with an odd symplectic
dg-manifold as the target space.

One may also think of an $\Li$-algebroid as a sheaf of $\Li$-algebras
acting infinitesimally on a smooth manifold, see a remark after
Definition~\ref{Li-algebroid}. This notion is essentially
K.~Costello's notion of an $\Li$ space \cite{costello}, see
\cite{grady-gwilliam} and another remark after
Definition~\ref{Li-algebroid} for details.

\subsection*{Formal graded and differential graded manifolds}

From now on we will be focusing on \emph{formal graded manifolds} $(M,
\hCi_V) := (M, S(\mathcal{V})^*)$, where $M$ is a graded maniold, $V$
a graded vector bundle over $M$ and the sheaf of functions
$S(\mathcal{V})^*)$ is the graded dual of the symmetric coalgebra with
respect to the shuffle coproduct. One can think of the algebra
$S(\mathcal{V})^*$ as an algebraic version of completion of the
algebra $S(\mathcal{V}^*)$. Any graded manifold $M$ may be regarded as
a formal manifold over itself associated with the zero vector bundle
over $M$.

A \emph{morphism $V \to W$ of formal graded manifolds} is a morphism
of locally ringed spaces $(M, S(\mathcal{V})^*) \to (N,
S(\mathcal{W})^*)$ induced by a vector bundle morphism
\[
\begin{CD}
S(V) @>>> W\\
@VVV @VVV\\
M @>>> N
\end{CD}
\]
This requirement may be understood as a continuity condition with
respect to our algebraic completions. A \emph{formal dg-manifold} is a
formal graded manifold endowed with a \emph{differential},
\emph{i.e}., a degree-one, $\nr$-linear derivation $d$ of the
structure sheaf satisfying $d^2 = 0$. %This differential is assumed to
% be determined by the following version of Costello's structure of an
%\emph{$\Li$-space} on the pullback of the vector bundle $V$ from $M$
%to the shifted tangent bundle $T[1]M$ regarded as a dg-manifold via
%Example \ref{derham}:
%\begin{quote}
%A codifferential $d$ on the graded coalgebra $S_\Omega (\Omega
%\otimes_A V)$ over the commutative dg-algebra $\Omega$, where $A :=
%\Ci(M)$ and $\Omega := \Ci(T[1]M) = \Omega^\bullet (M)$.
%\end{quote}
%Such a codifferential induces a differential on the $A$-linear dual
%$S_A (V)^* = \Hom_A (S_A(V),A)$ which is naturally identified with
%$\Hom_\Omega (S_\Omega (\Omega \otimes_A V), A)$, where $A$ acquires
%an $\Omega$-module structure from the canonical isomorphism
%$\Omega/\Omega^{\ge 1} = A$.

\emph{Morphisms of $($formal$)$ dg-manifolds} will have to respect
differentials. Since we work in the $\Ci$ category, we will routinely
substitute sheaves with spaces of their global sections.

A graded manifold $V$ over $M$ of any of the above flavors comes with
a morphism $V \to M$ given by the inclusion of $\Ci_M$ into the
function sheaf of $V$. There is also a \emph{relative basepoint} given
by the zero section of the vector bundle $V \to M$, which induces a
morphism $M \to V$ of (formal) graded manifolds over $M$. A
\emph{based morphism} must respect zero sections of the structure
vector bundles. The structure differential in the dg case will also be
assumed to be \emph{based}, that is to say, the zero section $M \to V$
must be a dg-morphism.

Let $V\to M$ be a graded vector bundle over a (possibly, graded)
manifold $M$. We will think of its total space as a \emph{pointed
  formal graded manifold $V[1]$, fibered over} $M$.  %It is determined
%by the graded cocommutative $C^\infty(M)$-coalgebra $\hCi (V[1]) :=
%\Gamma(M, S(V[1]))$ on $M$. The coalgebra structure is the standard
%shuffle coalgebra structure, making $S(V[1])$ a cofree graded
%cocommutative coalgebra on the vector bundle $V[1]$.
The term ``pointed'' is understood in a fiberwise (relative) sense and
refers to the fact that the fiber bundle $V[1] \to M$ has a canonical
zero section, given by a standard augmentation
$ \hCi (V[1]) \to \Ci (M)$. A \emph{dg structure} on the pointed
formal graded manifold $V[1]$ over $M$ is a choice of a square-zero,
degree-one derivation $d$ of the $\nr$-algebra $\hCi (V[1])$ such that
the zero section $\hCi (V[1]) \to \Ci (M)$ respects the differentials.
Here the differential on $\Ci (M)$ is assumed to be zero. %This
%condition is equivalent to the condition $\Ci(M) \subset \ker d$ on
%the codifferential $d$.

\begin{definition}
\label{Li-algebroid}
\noindent
\newline
\begin{enumerate}
\item An \emph{$L_\infty$-algebroid} is a graded vector bundle
  $V\to M$ with the structure of a dg-manifold on the pointed formal
  graded manifold $V[1]$ over $M$.  That is, it is the structure of a
  dg-algebra over $\Ci(M)$ on the graded commutative algebra
  $\Gamma (M, S(V[1])^*)$ such that the differential is compatible
  with the augmentation $\Gamma (M, S(V[1])^*) \to \Ci(M)$.
\item An \emph{$\Li$-morphism} of $L_\infty$-algebroids $V\to M$ and
  $W\to N$ is a formal pointed dg-manifold morphism
  $(V[1], d^V) \to (W[1], d^W)$.  Equivalently, it is an augmented
  dg-algebra morphism
  $\Gamma (N, S(W[1])^*) \to \Gamma (M, S(V[1])^*)$ over a graded
  algebra morphism $\Ci(N) \to \Ci(M)$.
  \end{enumerate}
 \end{definition}

\begin{example}
  When $M$ is a point, an $\Li$-algebroid is nothing but an
  $\Li$-algebra, and the notion of an $\Li$-morphism of Lie algebroids
  over a point reproduces the standard notion of an $\Li$-morphism of
  $\Li$-algebras.
\end{example}

\begin{remark}
  There is a generalized \emph{$\Li$-anchor map} associated to an
  $\Li$-algebroid $V \to M$. Indeed, the composition of its structure
  differential with the unit map
  $C^\infty (M) \to \Gamma (M, S(V[1])^*)$ gives a degree-one
  derivation with values in $\Gamma(M, \linebreak[0] S(V[1])^*)$. This
  gives rise to a $C^\infty(M)$-module morphism
  $\Omega^1(M) \to \Gamma(M, S(V[1])^*)$ by the universal property of
  K\"ahler differentials. It extends uniquely to a dg-algebra morphism
  $\Omega^\bullet(M) \to \Gamma(M, S(V[1])^*)$.  This, in its turn,
  induces a morphism of formal pointed dg-manifolds $V[1] \to T[1]M$,
  or an $\Li$-morphism of $\Li$-algebroids, generalizing the
  anchor. One can also think of it as an $\Li$-action of the
  $\Li$-algebroid $V$ on the base graded manfiold $M$.
\end{remark}

% \begin{remark}
%  If $V \to M$ is an $\Li$-algebroid, then
%  $\Gamma(M, \Omega^\bullet_M \otimes S(V[1])^*) = \Gamma (T[-1]M,
%  \linebreak[0] S(\pi^* (V[1]))^*)$,
%  where $\pi: T[-1]M \to M$ is the natural projection, is a
%  dg-commutative algebra over the de Rham algebra $\Omega^\bullet (M)$
%  and, in fact, an $\Li$-space in the sense of Costello
%  \cite{costello}, cf.\ the main result of \cite{grady-gwilliam} in
%  the case when $V$ is a Lie algebroid.
% \end{remark}

\subsection{The Poisson manifold approach}
% A more concise and convenient characterization of Lie algebroid
% morphisms can be given by invoking an analog of the Kostant-Kirillov
% Poisson bracket that can be defined on the dual vector bundle $V^*$ of
% a Lie algebroid $V\to M$. Moreover, on can also characterize Lie
% algebroid morphisms in Poisson terms.
The data of a Lie algebroid on $V\to M$ can also be cast in the form
of a Poisson structure on the linear dual bundle $V^*\to M$
generalizing the well-known Kostant-Kirillov (also known as the
Lie-Poisson) bracket defined on the linear dual of a Lie algebra.
More specifically, identifying smooth functions on $V^*$ constant
along the fibers with functions on $M$ and identifying functions on
$V^*$ linear along the fibers with sections of $V$, we set
 \[
  \{f,g\}_{V^*}=\begin{cases}
                 [f,g],\quad &f,g\in \Gamma(V)\\
                 \rho(f)g,\quad &f\in\Gamma(V), g\in C^{\infty}(M)\\
                 0,\quad &f,g\in C^\infty(M)
                \end{cases}
 \]
Extending this bracket further to the polynomial and smooth functions
via the Leibniz rule and completion endows $V^*$ with a well-defined
Poisson structure. Note that the corresponding Poisson tensor will be
linear along the fibers of $V^* \to M$.

\begin{theorem}[T.~J. Courant \cite{courant:Dirac}]
\label{courant}
Let $V \to M$ be a vector bundle. Then the following structures are
equivalent:
 \begin{enumerate}
  \item A Lie algebroid structure on $V \to M$;
  \item A Poisson structure on the total space of the vector bundle
    $V^* \to M$ such that the Poisson structure is linear along the
    fibers.
 \end{enumerate}
\end{theorem}

Two Poisson structures arising in this fashion on the linear duals of
Lie algebroids $V\to M$, $W\to N$ can be related by means of Lie
algebroid \textit{comorphisms} rather than morphisms. Namely, a Lie
algebroid comorphism from $V$ to $W$ over $f: M\to N$ is a
(base-preserving) morphism of vector bundles $\phi: f^*W \to V$ over
$M$ such that
$$\phi^\#([X,Y])=[\phi^\#(X), \phi^\#(Y)]$$
for all $X,Y\in \Gamma(W)$ and
$df\circ \rho_V \circ \phi^\# = \rho_W$.
Here, $\phi^\#$ is the natural composition
$\Gamma(W)\overset{f^*}{\to} \Gamma(f^*W)\overset{\phi_!}{\to}
\Gamma(V)$. Such a comorphism yields a vector bundle morphism
\[
 V^*\to (f^*W)^* \xrightarrow{\sim} f^*W^* \to W^*.
\]

\begin{theorem}[P.~Higgins, K.~Mackenzie \cite{HigMac}]
\label{Hig-Mac}
 Lie algebroid comorphisms from $V\to M$ to $W\to N$ are in one-to-one
 corresponence with vector bundle morphisms $V^*\to W^*$ that are
 Poisson.
\end{theorem}

\subsection{Lie coalgebroids}
\label{Lie-coalg}

Keeping in mind our main objective of studying Lie bialgebroids and
their homotopy generalization, we briefly describe the notion of a Lie
coalgebroid.

\begin{definition}
 A \textit{Lie coalgebroid} structure on a vector bundle $V\to M$ over a 
smooth manifold $M$ consists of
 \begin{itemize}
  \item 
  an $\nr$-linear mapping $\delta: \Gamma(V)\to \Gamma(V)
  \wedge_\nr \Gamma(V)$ (a \textit{Lie cobracket})
  satisfying the co-Jacobi identity
  \[
   \circlearrowright(\delta \otimes id)\delta = 0,
  \]
    where 
    $$\circlearrowright(x\otimes y\otimes z)=x\otimes y\otimes z+y\otimes z\otimes x+z\otimes x\otimes y;$$
  \item
  a vector-bundle morphism $\sigma: T^*M \to V$, called the
  \textit{coanchor}, subject to the co-Leibniz rule
 \end{itemize}
\[
\delta (fX)= f\delta(X) + \sigma(df) \wedge X, \quad X\in
\Gamma(V),\,f\in C^{\infty}(M).
\]
\end{definition}
This implies, in particular, that
\[
\delta(\sigma(\omega)) = (\sigma \wedge \sigma) (d \omega), \quad
\omega \in\Gamma(T^*M).
\]
That is, $\sigma$ induces a morphism $\Gamma(T^*M)\to \Gamma(V)$ of
Lie coalgebras.

\begin{theorem}
\label{coalgebroid}
Let $V\to M$ be a vector bundle. Then the structures of
\begin{enumerate}
\item a Lie coalgebroid on $V$,
\item a dg-manifold on the graded manifold $V^*[1]$
\end{enumerate}
are equivalent.
\end{theorem}

The proof of this theorem is a rather straightforward exercise on the
definitions; it is similar to the proof of Theorem~\ref{Vaintrob}.

\begin{examples}
\label{coalg-ex}
\noindent
\newline
\begin{enumerate}
 \item
\label{cotan-coalg} The cotangent bundle $T^*M\to M$ is
   trivially a Lie coalgebroid with the cobracket being the
   restriction of the de Rham differential $d_{dR}$ onto $T^*M$ and
   the coanchor $\sigma=id_{T^*M}$.
\item
\label{dual-coalg}
Any Lie algebroid structure on $V \to M$ gives rise to a Lie
coalgebroid structure on the linear dual bundle $V^*\to M$
\cite[Section 1.4.14]{beilinson-drinfeld}.\footnote{We remind the
  reader that we assume all vector bundles to be of finite rank.}  In
particular, taking $V$ to be the standard tangent Lie algebroid $TM\to
M$ recovers the previous example. Coversely, a Lie coalgebroid
structure on $V\to M$ induces a Lie algebroid structure on $V^* \to
M$.
\end{enumerate}
\end{examples}

The last example combined with the Poisson bracket construction
outlined in the previous section implies, in particular, that the
structure of a Lie coalgebroid on $V\to M$ induces a fiberwise linear
Poisson structure on $V$. More concretely, identifying, as before,
fiberwise linear functions on $V$ with the sections of $V^*$ and
fiberwise constant functions on $V$ with the elements of
$C^\infty(M)$, we get
\[
 \{\alpha, \beta\}_V =
 \begin{cases}
  (\alpha\otimes\beta)\delta,\quad &\alpha,\beta\in \Gamma(V^*)\\
  \alpha(\sigma(d\beta)),\quad &\alpha\in\Gamma(V^*),\beta\in
  C^{\infty}(M)\\
  0,\quad &\alpha,\beta\in C^{\infty}(M)
 \end{cases}.
\]
This is to be extended further via the Leibniz rule and completion. As an upshot, it enables us to give a 
concise definition of a Lie coalgebroid morphism.
\begin{definition}
\label{mor-coalg}
A \emph{morphism $V \to W$ of Lie coalgebroids} is a vector bundle
morphism $(V\to M)\to (W\to N)$ such that the map of total spaces $V
\to W$ is Poisson.
\end{definition}

% Note that this is 
% dramatically
% different from the notion of a morphism
% $(W^*\to N)\to (V^*\to M)$ between the dual Lie algebroids, unless $M
% = N$ and we consider morphisms lifting the identity map $\id: M \to
% M$. The latter is just true by Definition~\ref{Lie-morphism}. 

\begin{example}
The coanchor map $\sigma: T^*M \to V$ of a Lie coalgebroid $V\to M$ is
a Lie coalgebroid morphism, where $T^*M$ is given the standard
cotangent Lie coalgebroid structure as in
Example \ref{coalg-ex}\eqref{cotan-coalg}.  In this case, the Poisson
structure on $T^*M$ corresponding to the Lie coalgebroid structure is
the standard Poisson structure on the cotangent bundle.
\end{example}

\subsection{The odd Poisson manifold approach}
\label{odd-Poisson}

Yet another algebraic structure naturally associated with a Lie
algebroid $V\to M$ is defined on $\Gamma(\wedge^{\bullet} V)$.
Namely, there is a canonical extension of the Lie bracket on
$\Gamma(V)$ to a bracket
 \[
  [,]:\Gamma(\wedge^{k} V)\otimes\Gamma(\wedge^{l} V)\to
  \Gamma(\wedge^{k+l-1} V), \quad k,l\geq 1
 \]
of degree $-1$ satisfying the graded Jacobi and Leibniz rules.  For
$X\in\Gamma(V),f\in C^{\infty}(M)$, we set
 \[
  [X,f]=\rho(X)(f).
 \]
Altogether, this turns $\Gamma(\wedge^{\bullet} V)$ into a
Gerstenhaber (or an \textit{odd Poisson}) algebra. The converse is
also true:
 \begin{theorem}[A.~Vaintrob \cite{vaintrob}]
\label{Vaintrob-Poisson}
Let $V\to M$ be a vector bundle. The following structures are
equivalent:
 \begin{enumerate}
  \item
  A Lie algebroid structure on $V\to M$;
\item A Gerstenhaber algebra structure on $\Gamma(\wedge^{\bullet} V)$
  (taken with the standard multiplication);
\item A graded Poisson structure of degree $-1$ on $V^*[1]$.
 \end{enumerate}
 \end{theorem}

 This theorem can be regarded as an odd analogue of Courant's Theorem
 \ref{courant}. Likewise, the following statement is an odd analogue
 of Higgins-Mackenzie's Theorem \ref{Hig-Mac}.

\begin{proposition}
  There are natural bijections between the following sets:
 \begin{enumerate}
\item The set of Lie algebroid comorphisms from $V\to M$ to $W\to N$;
\item The set of Gerstenhaber algebra morphisms
  $\Gamma(\wedge^{\bullet} W)\to\Gamma(\wedge^{\bullet} V)$;
\item The set of graded Poisson manifold morphisms
  $V^*[1] \to W^*[1]$.
 \end{enumerate}
 \end{proposition}

 \begin{examples}
 \noindent
 \newline
  \begin{enumerate}
  \item For a Lie algebra $V$, $\Gamma(\wedge^\bullet V)$ is the
    underlying space of the homological Chevalley-Eilenberg complex
    with trivial coefficients. The odd Poisson bracket on
    $\Gamma(\wedge^\bullet V)$ is a derived bracket\cite{KS-derived}
    generated by the homological Chevalley-Eilenberg differential.
  \item
  For a tangent Lie algebroid $TM$, $\Gamma(\wedge^\bullet TM)$ is the 
Schouten-Nijenhuis algebra of multivector fields.
\item For a Lie algebroid associated with a Poisson manifold $M$ as in
  Example \ref{basic-examples}\eqref{5}, $\Gamma(\wedge^\bullet
  T^*M)=\Omega^\bullet(M)$ is the underlying space of the homological
  Poisson complex. The differential in this case (known as the
  \textit{Brylinski} differential) is
  $d=[i_\pi,d_{dR}]$.
 \end{enumerate}
 \end{examples}
 \begin{remark}
   The Poisson manifold $V^*$, the dg-manifold $V[1]$ and the odd
   Poisson manifold $V^*[1]$ determined by a Lie algebroid $V\to M$
   are known as $P$-, $Q$- and $S$-\textit{manifolds}, respectively,
   associated to $V$ \cite{ThVor}. In that regard, Lie bialgebroids
   (see Section~\ref{bi}) manifest themselves in the form of $QP$- or
   $QS$-manifolds, comprising a pair of such structures in a
   compatible way.
 \end{remark}

We also have analogous statements for Lie coalgebroids.

\begin{theorem}
\label{coald-odd-Poisson}
   Let $V\to M$ be a vector bundle. Then the structures of
\begin{enumerate}
\item a Lie coalgebroid on $V$,
\item a graded Poisson structure of degree $-1$ on $V[1]$
\end{enumerate}
are equivalent.
%\footnote{The structure of a dg-algebra on a graded
%  algebra boils down to a differential compatible with the
%  multiplication.} 
Furthermore, there are natural bijections between
the following sets:
\begin{enumerate}
\item The set of morphisms of Lie coalgebroids $V\to M$ and $W\to N$;
\item The set of graded Poisson morphisms $V[1] \to W[1]$.
\end{enumerate}
\end{theorem}

 \subsection{Connections and associated BV algebras}

 This section follows the paper \cite{PingXu} by Ping Xu. For a Lie
 algebroid $V\to M$, endowing the associated Gerstenhaber algebra
 $\Gamma(\wedge^\bullet V)$ with some extra data in the form of a
 differential operator of order one or two, subject to certain
 compatibility conditions, determines an additional structure on $V\to
 M$. The former case will be addressed in Section~\ref{bi}; to handle the
 latter case, we need the following
\begin{definition}
 Let $V\to M$ be a Lie algebroid and $E\to M$ be a vector bundle.
 A linear mapping $$\nabla:\Gamma(V)\otimes\Gamma(E)\to \Gamma(E),\quad 
X\otimes s\mapsto \nabla_X (s)$$ is called a \textit{$V$-connection} 
if
\begin{enumerate}[i.]
 \item 
 $\nabla_{fX}(s)=f\nabla_X(s)$;
 \item
 $\nabla_{X}(fs)=(\rho(X)f)s+f\nabla_X(s)$
\end{enumerate}
for all $f\in C^{\infty}(M), X\in\Gamma(V), s\in\Gamma(E)$.
\end{definition}

The \textit{curvature} of a $V$-connection $\nabla$ on $E\to M$ is an element 
$R\in \Gamma(\wedge^2 V^*)\otimes End(E)$ defined by
\[
 R(X,Y)=\nabla_X\nabla_Y-\nabla_Y\nabla_X-\nabla_{[X,Y]},\quad X,Y\in\Gamma(V),
\]
and the \textit{torsion} is
\[
 T(X,Y)=\nabla_X Y-\nabla_YX-[X,Y]\quad X,Y\in\Gamma(V).
\]
A $V$-connection is said to be \textit{flat} if $R\equiv0$.

If a Lie algebroid $V\to M$ is of rank $n$ as a vector bundle, then any 
$V$-connection of the canonical line bundle $E=\wedge^n V$ determines an 
operator $\Delta:\Gamma(\wedge^\bullet 
V)\to\Gamma(\wedge^{\bullet-1} V)$ on the Gerstenhaber algebra 
$\Gamma(\wedge^\bullet V)$:
\begin{multline*}
  \Delta \omega (X_1, \dots , X_{p+1})
  := \sum_i (-1)^{i-1}
  \nabla_{X_i} \omega (X_1, \dots, \hat{X}_i, \dots, X_{p+1}) \\
  + \sum_{i < j} (-1)^{i+j}\omega([X_i,X_j], \dots, \hat{X}_i, \dots,
  \hat{X}_j, \dots, X_{p+1}),
\end{multline*}
where $\omega \in \Gamma (\wedge^{n-p} V)$ is identified
with a section of $\Hom(\wedge^p V, \wedge ^n V)$.
\begin{lemma}
 $\Delta$ is a differential operator of order two.
\end{lemma}

\begin{theorem}[P.~Xu \cite{PingXu}]
\label{BV-conn}
There is a one-to-one correspondence between $V$-connections on
$E=\wedge^n V$ and linear operators $\Delta$ generating the bracket on
$\Gamma(\wedge^\bullet V)$. Under this correspondence, a flat
$V$-connection induces a square-zero differential operator $\Delta$ of
order two, thus turning $\Gamma(\wedge^\bullet V)$ into a
Batalin-Vilkovisky algebra.
\end{theorem}
Note that a flat $V$-connection on $E=\wedge^n V$ always exists.
\begin{examples}
 \noindent
 \newline
 \begin{enumerate}
 \item Given a Lie algebra $V$ of dimension $n$, the line bundle
   $\wedge^n V$ has a (flat) $V$-connection corresponding to the
   adjoint action of $V$ on $\wedge^n V$. The corresponding operator
   $\Delta$ is the homological Chevalley-Eilenberg operator.
\item
%\begin{sloppypar}
  The Brylinski differential $d$ on the homological Poisson complex
  $\Gamma(\wedge^\bullet T^*M) \linebreak[0] = \Omega^\bullet (M)$ of
  an $n$-dimensional Poisson manifold $M$ generates the corresponding
  bracket.  The flat connection on $\Omega^n_M \to M$ associated to
  $d$ by means of Theorem \ref{BV-conn} is given by
  $\nabla_\theta\omega=\theta\wedge di_\pi(\omega), \theta\in
  \Gamma(T^*M), \omega\in \Omega^n(M)$, \cite{Evens-Lu-Weinstein}.
%\end{sloppypar}
 \end{enumerate}
\end{examples}

\section{Lie bialgebroids}
\label{bi}

Let $V$ be a vector bundle over $M$ such that $V\to M$ is both a Lie
algebroid and a Lie coalgebroid. Denote by $d$ the coboundary operator
on $\Gamma(\wedge^\bullet V) = \Gamma (S(V[-1]))$ induced by the Lie
coalgebroid structure, as in Theorem~\ref{coalgebroid}.

\begin{definition}
\label{def-bi}
A vector bundle $V \to M$ with the structures of a Lie algebroid and a
Lie coalgebroid is a \textit{Lie bialgebroid} if $d$ is a derivation
of bracket:
\begin{equation}
\label{compatibility}
 d([X,Y])=[dX,Y]+[X,dY],\quad X,Y\in \Gamma(V).
\end{equation}
\end{definition}

Note that in the equation above, $dX, dY \in \Gamma (\wedge^2 V)$ and
the Lie bracket is extended from $V$ to $\wedge^\bullet V$ as a
biderivation of degree $-1$, see Section \ref{odd-Poisson}.

\begin{remark}
  In view of the equivalence between Lie coalgebroid structures on
  $V \to M$ and Lie algebroid structures on $V^* \to M$, see Example
  \ref{coalg-ex}\eqref{dual-coalg}, one may think of a Lie bialgebroid
  as a pair $(V,V^*)$ of Lie algebroids satisfying the compatibility
  condition \eqref{compatibility}. This is a more common point of
  view. However, we prefer to think of a Lie bialgebroid as two
  compatible structures on $V$ in this section for functoriality
  reasons. We will return to the former viewpoint later, when we
  discuss the Hamiltonian approach.
\end{remark}

\begin{theorem}
If $V$ is a \textit{Lie bialgebroid}, then so is $V^*$.
\end{theorem}

\begin{theorem}
\label{Gerst}
Let $V \to M$ be a vector bundle. The structures of
\begin{enumerate}
\item a Lie bialgebroid on $V$,
\item a dg-Gerstenhaber algebra on $\Gamma(\wedge^\bullet V^*)$,
\item a dg-Poisson manifold on $V[1]$ with Poisson bracket of degree
  $-1$
\end{enumerate}
are equivalent.
\end{theorem}

\begin{remark}
  The multiplication on $\Gamma(\wedge^\bullet V^*)$ and the
  differential are to be related via the Leibniz rule. Sometimes, dg
  Gerstenhaber algebras with this property are called \textit{strong},
  \cite{PingXu}.
\end{remark}

\begin{definition}
 A \textit{morphism $V \to W$ of Lie bialgebroids} is a vector bundle
 morphism $(V \to M) \to (W \to N)$ which is a morphism of Lie
 algebroids and a morphism of Lie coalgebroids.
\end{definition}

Taking into account Definition~\ref{mor-coalg}, the above definition
may immediately be reworded as follows.

\begin{proposition}
Let $V$ and $W$ be Lie bialgebrodis. A Lie algebroid morphism $V\to W$
is a morphism of Lie bialgebroids iff it is a Poisson map with respect
to the Poisson structures on $V$ and $W$ induced by the Lie
coalgebroid structures on $V$ and $W$, respectively, as in
Section~\ref{Lie-coalg}.
\end{proposition}

In the vein of Theorem~\ref{Gerst}, one obtains a characterization of
Lie bialgebroid morphisms:
\begin{theorem}
  There are natural bijections between the following sets:
 \begin{enumerate}
\item The set of Lie bialgebroid morphisms from $V$ to $W$;
\item The set of dg-Gerstenhaber algebra morphisms
  $\Gamma(\wedge^{\bullet} W^*)\to\Gamma(\wedge^{\bullet} V^*)$;
\item The set of dg-Poisson manifold morphisms $V[1] \to W[1]$.
 \end{enumerate}
\end{theorem}

\begin{examples}
\label{bi-examples}
 \noindent
 \newline
 \begin{enumerate}
 \item If $M$ is a point, a Lie bialgebroid over $M$ is a Lie
   bialgebra in the sense of Drinfeld.
\item   
  Let $M$ be a Poisson manifold with a Poisson bivector $\pi$,
  $T^*M\to M$ be the cotangent Lie algebroid associated to $\pi$ as in
  Example \ref{basic-examples}\eqref{5} with the canonical Lie
  coalgebroid structure as in Example
  \ref{coalg-ex}\eqref{cotan-coalg}. The Lichnerowicz differential
  $d_\pi=[\pi,-]$ is a derivation of the Schouten-Nijenhuis bracket on
  $\Gamma(\wedge^\bullet TM)$, thus giving the cotangent bundle $T^*M$
  a Lie bialgebroid structure.
  
  Conversely, let $V$ be a Lie bialgebroid over $M$. Then $\pi_V:=\rho
  \circ \sigma: T^*M\to TM$, where $\rho$, $\sigma$ are the anchor and
  the coanchor maps of $V\to M$, respectively, defines a Poisson
  structure on $M$.
\item 
\label{bi-3}
Suppose $V\to M$ is a Lie algebroid with an anchor $\rho$ and a
structure differential $d^V$ on $\Gamma(\wedge^\bullet V^*)$. Let
$r\in\Gamma(\wedge^2 V)$ be such that $[r,r]=0$. Denote by $r^\#$ the
associated bundle map $V^*\to V$.  One can show that
$\rho^*:=\rho\circ r^\#: V^*\to TM$ and
  \[
   [\phi,\psi] = L_{r^\# (\phi)}(\psi) - L_{r^\#(\psi)}(\phi) - d^V
   ((\phi \wedge \psi)r),\quad \phi,\psi\in \Gamma(V^*),
  \]
where
\[
L_X := [d^V, \iota_X], \quad X \in \Gamma(V),
\]
determine a Lie algebroid structure on $V^*\to M$.  Furthermore, a
simple check confirms that the pair $(V,V^*)$ is actually a Lie
bialgebroid. In particular, taking $V=TM$ recovers the previous
example.
  \item
A \textit{Nijenhuis structure} on a smooth manifold $M$ is a vector
bundle endomorphism $N:TM\to TM$ such that its \textit{Nijenhuis
  torsion}
  \[
   [N(X),N(Y)]-N([N(X),Y]+[X,N(Y)])+N^2([X,Y])
  \]
vanishes for any $X,Y\in \Gamma(TM)$.  A prototypical example of a
Nijenhuis structure arises in the form of a \textit{recursion
  operator} of an integrable bi-Hamiltonian system
\cite{magri,olver}. Namely, given a pair $\pi_0,\pi_1$ of Poisson
tensors on a manifold $M$ such that any linear combination $\lambda
\pi_0+\mu \pi_1$ is Poisson as well and $\pi_0$ is symplectic, then
$N:={\pi_1}^\#\circ (\pi_0^\#)^{-1}$ is a Nijenhuis structure on $M$.
  
A Nijenhuis structure on $M$ induces \cite{KS} a Lie algebroid
structure on $TM$ with the bracket
  \[
   [X,Y]_N:=[N(X),Y]+[X,N(Y)]-N([X,Y])
  \]
and $N:TM\to TM$ being an anchor map.  Now, if $M$ is a Poisson
manifold, then $T^*M$ can be given the Lie algebroid structure of
Example \ref{basic-examples}\eqref{5}. This produces a Lie coalgebroid
structure on $TM$ by Example \ref{coalg-ex}\eqref{dual-coalg}. It
turns out \cite{KS} that these two structures on $TM$ are compatible,
thereby making $TM$ a Lie bialgebroid, provided
  \[
   N\circ \pi^\# = \pi^\#\circ N^*,
  \]
and
  \[
   \{\alpha, \beta\}_{N\pi}=\{N^*(\alpha),\beta\}_\pi +
   \{\alpha,N^*(\beta)\}_\pi-N^*(\{ \alpha , \beta\}_\pi)
  \]
for all $\alpha,\beta\in\Gamma(T^*M)$.
  
  \item
\label{groupoid}
Recall that a Lie groupoid is a (small) groupoid
$s,t:\mathcal{G}\rightrightarrows M$ such that its set of objects $M$
and the set morphisms $\mathcal{G}$ are smooth manifolds, and the
source and the target maps $s,t$ along with the composition
$\mathcal{G}\times\mathcal{G}\to\mathcal{G}$, the unit $M\to
\mathcal{G}$ and the inverse map $\mathcal{G}\to\mathcal{G}$ are
smooth.  The source and target maps are also assumed to be
submersions.
  
Given a Lie groupoid $\mathcal{G}\rightrightarrows M$, we define the
associated Lie algebroid $V\to M$ as follows. As a vector bundle, $V=
\left. \ker (ds) \right|_M$, where the restriction is taken along
the unit map $M\to \mathcal{G}, x\mapsto 1_x$. A Lie bracket on
$\Gamma(V)$ is obtained by identifying the sections of $V$ with the
right-invariant vector fields on $\mathcal{G}$ and the anchor map is
$dt:T\mathcal{G}\to TM$ restricted onto $V\subset T\mathcal{G}$.
  
  A \textit{Poisson groupoid} is a Lie groupoid
  $\mathcal{G}\rightrightarrows M$ with a Poisson structure $\pi$ such
  that the graph of the composition
  $m:\mathcal{G}\times\mathcal{G}\to\mathcal{G}$ is a coisotropic
  submanifold of
  $\mathcal{G}\times\mathcal{G}\times\bar{\mathcal{G}}$, where
  $\bar{\mathcal{G}}$ denotes $\mathcal{G}$ with the opposite Poisson
  tensor $-\pi$.  One can show \cite{McK-Xu} that a Poisson structure
  on $\mathcal{G}$ induces a Lie algebroid structure on $V^*$.
  Furthermore, it is compatible with the Lie algebroid structure on
  $V\to M$, giving rise to a Lie bialgebroid over $M$.  This
  generalizes the well-known construction of Lie bialgebras arising as
  infinitesimal counterparts of Poisson-Lie groups.
 \end{enumerate}
\end{examples}

\section{The Hamiltonian approach}

Let $Q$ be a (graded) vector field on a graded manifold $V$.  The
\textit{cotangent} (or \textit{Hamiltonian}) lift
\begin{align*}
\Gamma(V, TV) & \to C^\infty (T^*V),\\
Q & \mapsto  \mu_Q,
\end{align*}
is defined by setting
  \begin{equation}
\label{linear}
   \mu_Q(x,p)=p(Q_x),\quad p\in T_x^*V.
  \end{equation}
In local Darboux coordinates $(x^*_i, x^i)$ on $T^*V$, if $Q = \sum_i
Q^i(x) \del/\del x^i$, then $\mu_Q = \sum_i Q^i(x) x^*_i$, a function
linear along the fibers of $\pi: T^*V \to V$.

\begin{remark}
  This construction has been rediscovered in the case where $M$ is a
  point by C.~Braun and A.~Lazarev \cite{braun-lazarev:unimodular}
  under the name of \emph{doubling}, see also \cite{kravchenko}.
\end{remark}

  \begin{proposition}
Let $Q_1,Q_2,Q$ be vector fields on $V$. Then
  \begin{enumerate}
  \item 
   $\{\mu_{Q_1},\mu_{Q_2}\}=\mu_{[Q_1,Q_2]}$
   \item
   $\pi_* (\{\mu_Q,-\})=Q$.
  \end{enumerate}
  \end{proposition}

Note that if $[Q,Q] = 0$, the proposition implies
$\{\mu_{Q_1},\mu_{Q_2}\}= 0$.

The case we are interested in corresponds to the graded manifold
$V[1]$ associated with a vector bundle $V\to M$.  By Theorem
\ref{Vaintrob}, a vector field $Q$ of degree $+1$ and such that
$[Q,Q]=0$ determines a Lie algebroid structure on $V\to M$.  This
leads to the following chain of correspondences:
  \hspace{-0.4in}
  \begin{align*}
    \left\{
   \begin{array}{c}
   \text{Lie algebroid}\\ \text{ structures }\\
   \text{on }V\to M
  \end{array}\right\}
  %\mathlarger{\mathlarger{\leftrightsquigarrow}}
  \leftrightsquigarrow
    \left\{
  \begin{array}{c}
   \text{Homological}\\
   \text{vector fields }Q\\
   \text{of degree +1}\\
   \text{on }V[1]
  \end{array}\right\}
  %\mathlarger{\mathlarger{\leftrightsquigarrow}}
  \leftrightsquigarrow 
   \left\{
  \begin{array}{c}
   \text{``Integrable odd}\\
   \text{Hamiltonians'':}\\
   \{\mu_Q,\mu_Q\}=0
  \end{array}\right\}
 \end{align*} 

\begin{theorem}[D.~Roytenberg \cite{roytenberg}]
  Lie algebroid structures on $V\to M$ are in one-to-one
  correspondence with functions $\mu$ on $T^*V[1]$ which are linear
  along the fibers of $T^*V[1] \to V[1]$, of total degree one in the
  natural $\nz$-grading on $C^\infty(T^*V[1])$ and such that
  $\{\mu,\mu\}=0$.
 \end{theorem}
 
 In terms of local Darboux coordinates $(x^i, \xi^a, x_i^*, \xi^*_a)$,
 where
 \begin{itemize}
  \item 
   $\{x^i\}$ are coordinates on $U\subset M$,
  \item
  $\{e_a\}$ is a basis of sections of $V$ over $U$,
  \item $\{\xi^a\}$ are the corresponding generators of $\Gamma(U,
    \wedge^\bullet V^*)$,
  \item 
   $\{x^*_i\}$ are coordinates along $T^*U \to U$,
  \item 
   $\{\xi^*_a\}$ are coordinates along $T^*V[1]|_U \to V[1]|_U$,
 \end{itemize}
 the above correspondence takes the following form:
   \begin{align*}
    \left\{
    \begin{array}{c}
    \rho(e_a)=A^{i}_a (x) \frac{\partial}{\partial x^i}\\
    {[}{e_a},{e_b}{]}=C^c_{ab}(x){e_c}
    \end{array}
   \right\}
   &\leftrightsquigarrow
    \left\{
  \begin{array}{c}
   d =\xi^a A^{i}_a(x)\frac{\partial}{\partial 
x^i}+\frac{1}{2}C^c_{ab}(x)\xi^a\xi^b\frac{\partial}{\partial \xi^c}
   \end{array}
   \right\}\\
   &\leftrightsquigarrow
   \left\{
  \begin{array}{c} 
  \mu =\xi^a A^{i}_a(x)x_i^*+\frac{1}{2}C^c_{ab}(x)\xi^a\xi^b\xi_c^*
  \end{array}\right\}
 \end{align*}
\begin{examples}
\noindent
\newline
 \begin{enumerate}
  \item 
  The Hamiltonian of a tangent Lie algebroid $TM\to M$ is 
$\mu=\xi^i x_i^*$.
  \item
  A Hamiltonian $\mu=\xi^a
  A^{i}_ax_i^*+\frac{1}{2}C^c_{ab}\xi^a\xi^b\xi_c^*$ with
  coordinate-independent structure coefficients $C^c_{ab}$ corresponds
  to an action Lie algebroid, \emph{cf}.\ Example
  \ref{basic-examples}\eqref{3}. Namely, in that case, $C^c_{ab}$ are
  the structure constants of the Lie algebra $\mathfrak{g}$ acting on
  $M$ and $A^i_a (x)$ are the coefficients of the anchor map
  $\rho(e_a)=A^i_a\frac{\partial}{\partial x^i}$.
 \item
 For a Poisson manifold $M$ with a Poisson bivector 
$\pi=\pi^{ij}\frac{\partial}{\partial x^i}\wedge \frac{\partial}{\partial 
x^j}$,
 the Hamiltonian of the corresponding Lie algebroid $T^*M\to M$ is
 $$
 \mu= \sum_{i,j} \xi^j \pi^{ij}x_i^*+ \sum_{i,j,k} \frac{1}{2}\frac{\partial}{\partial 
x^k}(\pi^{ij})\xi^i\xi^j\xi_k^*.
 $$
 \end{enumerate}
\end{examples}

% \begin{proposition} 
%   \label{alg-morphisms}
%   Let $\mu\in C^\infty(T^* V[1])$ and $\nu\in C^\infty(T^* W[1])$ be
%   the Hamiltonians determining Lie algebroid structures on $V\to M$
%   and $W\to N$, respectively.  Then Lie algebroid morphisms $V\to W$ are in 
% one-to-one correspondence with graded
%   manifold morphisms $f: V[1]\to W[1]$ such that
% \[
% F^*(\nu)= \Phi^*(\mu),
% \]
% where $F:f^*(T^* W[1])\to T^* W[1]$ is induced by the pullback of
% $T^* W[1]\to W[1]$ along $f$ and $\Phi: f^*(T^* W[1])\to T^* V[1]$ is
% the morphism of graded vector bundles over $V[1]$ given as the dual of
% the bundle map $TV[1] \to f^*(TW[1])$ induced by the differential $df:
% TV[1] \to TW[1]$.
% \end{proposition}
We would like to give a Hamiltonian characterization of Lie algebroid
morphisms.  To this end, consider Lie algebroids $V\to M$, $W\to N$.
We let $F: f^*(T^* W[1]) \to T^* W[1]$ be the pullback of
$f: V[1] \to W[1]$ along the projection $T^* W[1]\to W[1]$, and
$\Phi: f^*(T^* W[1])\to T^* V[1]$ be the morphism of graded vector
bundles over $V[1]$ given as the dual of the bundle map
$TV[1] \to f^*(TW[1])$ induced by the differential
$df: TV[1] \to TW[1]$.
\begin{proposition} 
 \label{alg-morphisms}
 Let $\mu\in C^\infty(T^* V[1])$, $\nu\in C^\infty(T^* W[1])$ be the
 Hamiltonians corresponding to the Lie algebroid structures on $V\to
 M$, $W\to N$.  Then Lie algebroid morphisms $V\to W$ are in
 one-to-one correspondence with graded manifold morphisms $f: V[1]\to
 W[1]$ such that
 \[
 F^*(\nu)= \Phi^*(\mu).
 \]
\end{proposition}

\begin{proof}
  We have to show that, given Hamiltonians $\mu, \nu$, the condition
  above is equivalent to the homological vector fields $Q^V=dp
  (\{\mu,-\})$ and $Q^W=dp(\{\nu,-\})$ being $f$-related. That is,
\begin{equation*}
%\label{1}
 %p_*(\{\mu,-\})(f^*\phi)=f^*(p_*\{\nu,-\}(\phi))
d f (Q^V_v) = Q^W_{f(v)}
\end{equation*}
for any $v\in V[1]$. 
%Here, $p$ denotes, by a slight abuse of 
%notation, each of the projections $T^*V[1]\to V[1]$ and $T^*W[1]\to W[1]$.

Indeed, for any point $(v,\zeta)$ in $f^*(T^* W[1])$, we have
\begin{align*}
 (F^*(\nu))(v,\zeta) = (\nu\circ F)(v,\zeta)=\nu(f(v),\zeta)=\zeta(Q^W_{f(v)}).
\end{align*}
On the other hand, 
\begin{align*}
 (\Phi^*(\mu))(v,\zeta)=(\mu\circ 
\Phi)(v,\zeta)=\mu(v,f^*(\zeta))=(f^*(\zeta))(Q^V_v)=\zeta(df(Q^V_v)).
\end{align*}
Thus, the equation $F^*(\nu) = \Phi^*(\mu)$ is equivalent to
$\zeta(Q^W_{f(v)}) = \zeta(df(Q^V_v))$ for all $v$ and $\zeta$.
\end{proof}

Given a pair $(V,V^*)$ of Lie algebroids over $M$, we get the
corresponding pair of Hamiltonians
 \[
  \mu\in C^{\infty}(T^* V[1])\text{ and }\mu_* \in C^{\infty}(T^*
  V^*[1])
\]
of degree one in the $\nz$-grading. We can bring them together in
$C^{\infty}(T^* V[1])$ by means of a canonical symplectomorphism
\[
L: T^*[2]V[1] \to T^*[2] V^*[1],
\]
known as the \emph{Legendre transform}. In local
coordinates, it reads
\[
(x, \xi, x^*, \xi^*) \mapsto (x, \xi^*, x^*, \xi)
\]
in our $\nz$-graded setting, see \cite[Section 3.4]{roytenberg} for
the $\nz/2\nz$-graded case.  Note that $V^*[1] = (V[1])^*[2]$ and the
shift along the cotangent directions is needed to make sure that $L$
respects grading. After the shift, the Poisson bracket $\{,\}$
acquires degree $-2$, and both $\mu$ and $\mu_*$ become elements of
degree $3$ in $C^{\infty}(T^*[2] V[1])$ and $C^{\infty}(T^*[2]
V^*[1])$, respectively.

% Note that $L^* \mu_*$, as
% a function on $T^*[2] V[1]$, is quadratic along the fibers of $T^*[2]
% V[1] \to V[1]$ and represents the Poisson tensor corresponding to the
% Lie algebroid structure on $V^*$.
\begin{lemma}
\label{Poisson}
Let $\mu_* \in C^{\infty}(T^*[2] V^*[1])$ be a Hamiltonian
corresponding to a Lie algebroid structure on $V^*$.  Then
  $$L^* \mu_*(v,\zeta)=\pi_{V^*,v}(\zeta,\zeta),\quad (v,\zeta)\in T^*[2]V[1]$$
  where $\pi_{V^*}$ is the graded Poisson tensor induced by the Lie
  algebroid structure on $V^*$ and $\pi_{V^*,v}$ is its value at $v\in V[1]$.
\end{lemma}

\begin{proof}
 In local coordinates we may write  $(v,\zeta)=(x^i, \xi^a, x_i^*, \xi_a^*)$,
 and $L(v,\zeta)=(x^i, \xi_a^*, x_i^*, \xi^a)$.
 Then
 \[
  (L^* \mu_*)(v,\zeta)=\mu_*(L(v,\zeta))=\xi^*_a A^{i}_a(x)x_i^*+\frac{1}{2}C_c^{ab}(x)\xi^*_a\xi^*_b\xi^c,
 \]
 where $A^i_a(x)$ and $C_c^{ab}(x)$ are the structure functions of $\mu_*$.
 
 On the other hand, 
 \[
 \pi_{V^*,v}(\zeta,\zeta)=\sum_{a,b}\pi_{V^*,v}(d\xi^a,d\xi^b) +
 \sum_{i,j}\pi_{V^*,v}(dx^i,dx^j)+\sum_{a,i}\pi_{V^*,v}(d\xi^a,dx^i),
 \]
 where the first and the last summands contribute
 $\frac{1}{2}C_c^{ab}(x)\xi^*_a\xi^*_b\xi^c$ and $\xi^*_a
 A^{i}_a(x)x_i^*$, respectively, and the second one is identically
 zero.
\end{proof}

  \begin{theorem}[D.~Roytenberg \cite{roytenberg}]
A pair $(V,V^*)$ of Lie algebroids is a Lie bialgebroid if and only if
$$\{\mu+L^*\mu_*,\mu+L^*\mu_*\}=0.$$
  \end{theorem}

  \begin{corollary}
\label{Lie-bialg}
    A structure of a Lie bialgebroid on a vector bundle $V \to M$ is
    equivalent to a Hamiltonian $\chi$ on $T^*[2]V[1]$, which is
    linear-quadratic along the fibers of $T^*[2]V[1] \to V[1]$ and
    is of degree three in the natural $\nz$-grading on functions
    on $T^*[2]V[1]$, and such that $\{\chi,\chi\}=0$.
\end{corollary}

\begin{example}
  For a Lie bialgebroid $(T^*M, TM)$ associated to a Poisson
  manifold $M$ as in Example \ref{bi-examples}(2), the Hamiltonian
  $\chi$ on $T^*[2]T^*M[1]$ is given by
\[
\chi = \sum_i \xi_i^* x_i^* +  \sum_{i,j} \xi^j \pi^{ij}x_i^*+ \sum_{i,j,k} \frac{1}{2}\frac{\partial}{\partial 
x^k}(\pi^{ij})\xi^i\xi^j\xi_k^*.
\]
\end{example}

\begin{theorem}
\label{Lie-bialg-mor}
Lie bialgebroid morphisms $V\to W$ are in one-to-one correspondence
with formal graded manifold morphisms $f: V[1]\to W[1]$ such that
\[
F^*(\psi)= \Phi^*(\chi),
\]
where and $\chi, \psi$ are the Hamiltonians on $T^*[2]V[1]$ and
$T^*[2]W[1]$ corresponding, respectively, to the given Lie bialgebroid
structures and $F, \Phi$ are as in Proposition~$\ref{alg-morphisms}$.
\end{theorem}

\begin{proof}
  As in the proof of Proposition~\ref{alg-morphisms}, at any point
  $(v,\zeta)$ in $f^*(T^*[2] W[1])$, the value $(F^*(\psi))(v,\zeta) =
  \psi(f(v),\zeta)$ is given by
\begin{enumerate}
\item evaluating the linear in $\zeta \in T^*_{f(v)}[2] W[1]$ part
  $\nu$ of $\psi = \nu + L^* \varepsilon$ given by the vector field
  $Q^W_{f(v)}$ on $W[1]$ as a linear functional on the cotangent
  bundle $T^*[2]W[1]$, see Equation~\eqref{linear},
\item evaluating the quadratic part $L^* \varepsilon$, which by
  Lemma~\ref{Poisson} is given by the Poisson tensor on $W[1]$
  corresponding to the Lie algebroid structure on $W^*$, and
\item adding the results together.
\end{enumerate}

Likewise, the value $(\Phi^*(\chi)) (v,\zeta) = \chi(v,f^*(\zeta))$ at
$(v,\zeta)$ is the sum of the values at $f^*(\zeta)$ of the linear
part $\mu$ of $\chi = \mu + L^*\mu_*$ given by the vector field
$Q^V_v$ and the quadratic part $L^*\mu_*$ given by the Poisson tensor
on $V[1]$.  Since equality of polynomial functions is equivalent to
equality of their homogeneous parts, the agreement of the functions
$F^*(\psi)$ and $\Phi^*(\chi)$ implies that $f$ is a morphism of Lie
algebroids respecting the graded Poisson structures.
\end{proof}

\section{$L_\infty$-bialgebroids}

Corollary~\ref{Lie-bialg} and Theorem~\ref{Lie-bialg-mor} motivate the
following $L_\infty$ generalizations of the notions of a Lie
bialgebroid and a Lie-bialgebroid morphism.

\begin{definition}
  An $\Li$-\emph{bialgebroid} over a (graded) manifold $M$ is a graded
  vector bundle $V \to M$ along with a degree-three function $\chi$ on
  the pointed formal graded manifold $T^*[2]V[1]$ such that
\begin{itemize}
\item $\{\chi,\chi\}=0$, \emph{i.e}., $\chi$ is an \emph{integrable
  Hamiltonian};
\item 
  $\chi$ vanishes on the zero section $V[1] \subset T^*[2]V[1]$ of the
  vector bundle $T^*[2]V[1] \to V[1]$ as well as on the restriction
  $T^*[2]V[1]|_M$ of this bundle to the zero section $M \subset V[1]$
  of the vector bundle $V[1] \to M$.
\end{itemize}
\end{definition}

Removing the second condition leads to an $L_\infty$ generalization of
the notion of a \emph{quasi-Lie bialgebroid}, \cite{roytenberg}, also
known as a \emph{curved Lie bialgebroid}, \cite{grady-gwilliam}.

\begin{remark}
  A seemingly natural attempt to define the notion of an
  $L_\infty$-bialgebroid in a way similar to Definition \ref{def-bi},
  as a pair of $L_\infty$-algebroids and $\Li$-coalgebroid structures
  on $V$ subject to some compatibility conditions, would be too
  restrictive, as such a structure would fail to comprise higher
  $L_\infty$ operations with multiple inputs and multiple outputs,
  cf.\ Examples \ref{lastone} and \ref{infty-bialgebras}. However,
  defining the notion of an $\Li$-bialgebroid via \emph{Manin
    $\Li$-triples}, as an $\Li$-algebroid structure on $V \oplus V^*$
  under some finite-rank conditions, should be possible, see
  \cite{kravchenko}, where this is done for $\Li$-bialgebras,
  \emph{i.e}., when $M$ is a point.
\end{remark}

\begin{example}
\label{lastone}
This example generalizes triangular Lie bialgebras in the sense of
Drinfeld \cite{Dr2}.

A \textit{generalized} (or \textit{higher}) Poisson structure on a
graded manifold $M$ is a (total) degree-two multivector field $P\in
\Gamma(S(T[-1]M))$ such that $[P,P]=0$, where the bracket is the
standard Schouten bracket.  As shown by H.~Khudaverdian and
Th.~Th. Voronov \cite{ThVor2}, such a structure induces $L_\infty$
brackets on the algebra of smooth functions $C^{\infty}(M)$ and on the
de Rham complex $\Gamma(S(T^*[-1]M))$ of $M$. These higher brackets
are known as the \textit{higher Poisson} and \textit{higher Koszul}
brackets, respectively.  The former generalizes the standard Poisson
bracket construction, while the latter generalizes Example
\ref{basic-examples}\eqref{5}, see also \cite{bruce2010}.

Pursuing these ideas in the direction of Example
\ref{bi-examples}\eqref{bi-3}, we start with a graded manifold $M$ and
a graded Lie algebroid $V\to M$, determined by a Hamiltonian
$\mu$. Let $r \in \Gamma(S(V[-1])) \subset \Ci (V^*[1])$ be a
degree-two element such that $[r,r] = 0$, where the bracket is the
degree-$(-1)$ Poisson bracket on $V^*[1]$ induced by the Lie algebroid
structure on $V$, as described in Section \ref{odd-Poisson}. Then the
following sequence of maps takes place:
  \[
  \alpha: \Ci(V^*[1])\to \Gamma(V^*[1], TV^*[1])\to \Ci(T^*V^*[1]).
  \]
Here, the first mapping associates the Hamiltonian vector field
$[f,-]$ to a function $f$, using the odd Poisson bracket on $V^*[1]$,
and the second one is the cotangent lift.  Each of these morphisms
respects the brackets, thus letting $r$ pass to a degree-one element
$\alpha(r)$ such that $\{\alpha(r),\alpha(r)\}=0$ on
$T^*V^*[1]$. After the degree shift to $T^*[2]V^*[1]$, the element
$\alpha(r)$ acquires degree 3.  Altogether, as a Hamiltonian, the sum
$\mu + L^*(\alpha(r))$ determines an $L_\infty$-bialgebroid structure
on $V$.
  
Compatibility of the coalgebroid component $L^*(\alpha(r))$ with $\mu$
becomes more apparent upon recognizing that
$L^*(\alpha(r))=\{\mu,L^*(r)\}$, where $\{,\}$ is the canonical
bracket on $T^*[2]V[1]$; see \cite{KS3}.  The $L_\infty$-algebroid
part on $V$ is a just a graded Lie algebroid, and there are no higher
mixed operations. A construction giving nontrivial higher brackets,
higher cobrackets and higher mixed operations along these lines in the
case of $M$ being a point can be found in \cite{BaVo2}.
\end{example}

\begin{definition}
\label{semistrict}
A \emph{semistrict $L_\infty$-morphism} $V \to W$ of
$L_\infty$-bialgebroids $V \to M$ and $W \to N$ is a morphism
$f: V[1] \to W[1]$ of pointed formal graded manifolds relating the
Hamiltonians on the shifted cotangent bundles $T^*[2]V[1]$ and
$T^*[2]W[1]$ in the sense of Proposition \ref{alg-morphisms}. The
``pointed'' condition means that $f$ maps the zero section of
$V[1] \to M$ to the zero section of $W[1] \to N$.
\end{definition}

Even though this definition looks like a direct generalization of
Theorem~\ref{Lie-bialg-mor}, which gives a Hamiltonian charaterization
of morphisms of Lie bialgebroids, we have chosen to use the word
\emph{semistrict}, because the morphism $V \to W$ of
$L_\infty$-coalgebroids under this definition is strict, while the
morphism $V \to W$ of $L_\infty$-algebroids may have ``higher,''
\emph{i.e}., $\Li$, components. See also Example
\ref{infty-bialgebras} below.

The definition of a full-fledged $\Li$-morphism beautifully overlaps
with the idea of deformation quantization, whence we need to introduce
a formal quantization parameter $\hbar$, to which it would be
convenient to assign a degree 2 in our graded context. It is common in
deformation quantization to consider functions on the cotangent bundle
polynomial in the momenta, see
\emph{e.g}. \cite{bordemann-neumaier-waldmann}, but the presence of a
formal parameter will allow us to consider formal series in the
momenta.

Let $V$ be a graded vector bundle over a graded manifold $M$ and $\chi
\in \hCi (T^*[2] V[1])$ an integrable Hamiltonian defining an
$\Li$-bialgebroid structure on $V$. We want to define an action of
$\chi$ on functions $g \in \hCi (V[1])$ on $V[1]$ by differential
operators.

The differential $dg$ of $g$ defines a section of the cotangent bundle
$T^* (V[1])$. The differential $\dv \chi$ along the vertical
directions of the vector bundle $T^*[2]V[1] \to V[1]$ is a section
of the relative cotangent bundle $T^*(T^*[2]V[1]/V[1]) \to
T^*[2]V[1]$, whose restriction to the zero section $V[1] \subset
T^*[2]V[1]$ may be identified with the shifted tangent bundle
$T[-2]V[1]$. (Formally speaking, there should be a double dual, but
given our assumption of finite dimensionality of each graded
component, the corresponding tangent bundle will be reflexive.) Using
the natural paring between the tangent and the cotangent bundles of
$V[1]$, we can pair $\left. (\dv \chi)\right|_{V[1]}$ and $dg$,
producing a function on $V[1]$, denoted
\[
\chi(g)_1 := \langle (\dv \chi)|_{V[1]}, dg \rangle.
\]
Note that this pairing will shift the degree by $-2$, given that
$\left. (\dv \chi) \right|_{V[1]}$ is a section of $T[-2] V[1]$.

Extend this pairing to pairings of degree $-2k$ between iterated
differentials $\left. (\dv^k \chi) \right|_{V[1]}$ and $d^k g$, viewed
as sections of the $k$th symmetric powers $S^k(T[-2]V[1])$ and
$S^k T^* V[1]$, respectively, for all $k \ge 1$:
\[
\chi(g)_k := \langle (\dv^k \chi)|_{V[1]}, d^k g \rangle.
\]
Note that because of the vanishing condition $\chi|_{V[1]} = 0$, the
$k=0$ term $\chi(g)_0$ will automatically be zero. Finally, define a
differential operator
\begin{align}
\notag
\Ci (V[1]) &\xrightarrow{\chi} \Ci(V[1]) [ \! [ \hbar] \! ],\\
\label{action}
  g & \mapsto \chi(g) := \sum_{k = 1}^\infty \hbar^{k-1} \chi(g)_k.
\end{align}
%in which the infinite summation will make sense, because for
%homogeneous $\chi$ and $g$, different terms will belong to graded
%components of different degrees.
Since the degree of $\chi$ as a function on $T^*[2]V[1]$ is three and
$\abs{\hbar} = 2$, the degree of $\chi$ as an operator on functions on
$V[1]$ is one, \emph{i.e}.,
\[
\abs{\chi(g)} = \abs{g} + 1.
\]
Note that the vanishing condition $\chi|_{T^*[2]V[1]|_M} = 0$ on the
Hamiltonian implies that $\chi(g)|_M = 0$, whereas the integrability,
$\{\chi, \chi\} = 0$, yields $\chi(\chi(g)) = 0$ or $\chi \circ \chi =
0$.

This action has a more familiar form in coordinates. Let $q^i$'s
denote coordinates on $V[1]$ and $p_i$'s denote the conjugate momenta,
resulting in Darboux coordinates $(p_i, q^i)$ on $T^*[2]V[1]$. Under
stronger assumptions of finite dimensionality, such as the finiteness
of the total rank of $V$ and the total dimension of $M$, so as the
index $i$ takes finitely many values, we would have
\[
\chi(g) = \sum_{k = 1}^\infty \frac{\hbar^{k-1}}{k!} \sum_{i_1, \dots,
  i_k} \pm \left. \frac{\del^k \chi}{\del p_{i_1} \dots \del p_{i_k}}
\right|_{V[1]} \frac{\del^k g}{\del q^{i_1} \dots \del q^{i_k}},
\]
where $\pm$ is a suitable Koszul sign.

Compare this with the star product of standard type:
\[
\chi * g := \sum_{k = 0}^\infty \frac{\hbar^k}{k!} \sum_{i_1, \dots,
  i_k} \pm \frac{\del^k \chi}{\del p_{i_1} \dots \del p_{i_k}}
\frac{\del^k g}{\del q^{i_1} \dots \del q^{i_k}},
\]
which in our case, when $g$ is a function on the base $V[1]$ of the
cotangent bundle, coincides with the star product of Moyal-Weyl type,
see \cite{bordemann-neumaier-waldmann}. Since $\hbar$ has degree 2,
the above star product is homogeneous.

We will be considering smooth maps
\[
f: V[1] \to S(W[1]),
\]
with which we would like to associate certain linear maps
\[
C^\infty (W[1]) \to C^\infty(V[1]).
\]
A map $f: V[1] \to S(W[1])$ induces a morphism
\[
f^*: C^\infty (S(W[1])) \to C^\infty(V[1])
\]
of algebras of smooth functions.  On the other hand, starting from a
smooth function $g \in C^\infty(W[1])$, we can build a linear (and
thereby, smooth) function on $S(W[1])$ by taking a formal Taylor
series
\[
T(g) := \sum_{k = 0}^\infty \frac{1}{k!} \left. \dv^k g\right|_M,
\qquad \left. \dv^k g \right|_M \in \Gamma (M, S^k(W[1])^*).
\]
This produces a linear map
\[
T: \Ci (W[1]) \to \Gamma (N, S(W[1])^*) \subset C^\infty(S(W[1])).
\]

\begin{definition}
\label{morphism}
An \emph{$L_\infty$-morphism} $V \to W$ of $L_\infty$-bialgebroids
$V \to M$ and $W \to N$ is a morphism
\[
f: V[1] \to S(W[1])
\]
of pointed formal graded manifolds relating the Hamiltonians $\chi_V$
and $\chi_W$ on the shifted cotangent bundles $T^*[2]V[1]$ and
$T^*[2]W[1]$, respectively, as follows:
\[
\chi_V \circ f^* \circ T = f^* \circ T \circ \chi_W,
\]
where the structure Hamiltonians $\chi_V$ and $\chi_W$ are regarded as
operators on functions via the action \eqref{action} and the condition
above is understood as an equation on linear operators
\[
C^\infty(W[1]) \to C^\infty(V[1])[\![\hbar]\!].
\]
\end{definition}

\begin{example}
%  \noindent
%  \newline
%  \begin{enumerate}
%  \item
\label{infty-bialgebras}
% \begin{sloppypar}
If $M$ is a point, $T^*[2]V[1]$ can be identified, as a graded
manifold, with $V^*[1]\oplus V[1]$ equipped with a Poisson bracket of
degree $-2$ known as the \textit{big bracket},
\cite{KS-derived,kravchenko}. The algebra of smooth functions on
$V^*[1]\oplus V[1]$ may be thought of as a completion of $S(V[-1]
\oplus V^*[-1])$, and it encodes various Lie and co-Lie operations on
$V$.  Schematically, the big bracket of two homogeneous tensors $f$
and $g$, interpreted as linear maps between graded symmetric powers of
$V[\pm 1]$, may be depicted as follows:
\begin{center}
\includegraphics[scale=0.8]{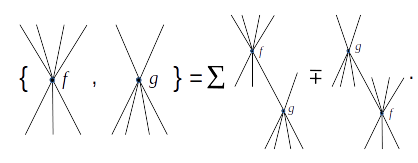}
\end{center}
Here, the summation is done over all possible ways to form 
an input-output pair for $f$ and $g$; the relevant signs are suppressed.

Now, an integrable Hamiltonian on $T^*[2]V[1]$ is a degree-three
function $\chi$ on $V^*[1]\oplus V[1]$ satisfying $\{\chi, \chi\} =
0$.  The condition that $\chi$ vanishes on $V[1]$ and $V^*[1]$ implies
that $\chi$ belongs to a functional completion of $S^{>0}(V[-1])
\otimes S^{>0}(V^*[-1])$, thereby resulting in a notion of an
$L_\infty$-bialgebra structure on $V$, equivalent, up to completion,
to Kravchenko's notion \cite{kravchenko} of an
$L_\infty$-bialgebra. We will adopt the algebraic version of
\cite{BaVo2} and assume $\chi \in \prod_{m,n \ge 1} \Hom (S^m (V[1]),
S^n (V[-1]))$ of degree three and satisfying the Maurer-Cartan
equation $\{\chi, \chi\} = 0$.
%\end{sloppypar}

The notion of an $L_\infty$-morphism for two $L_\infty$-bialgebroids,
see Definition~\ref{morphism}, leads to the following version of the
notion of an $L_\infty$-morphism of two $L_\infty$-bialgebras $V$ and
$W$, regarded as $L_\infty$-algebroids over a point.  An
$L_\infty$-\emph{morphism $V \to W$ of $L_\infty$-bialgebras} is a
linear map $\Phi: S(V[1]) \to S(W[1])$ commuting with the solutions
$\chi_V \in \prod_{m,n \ge 1} \Hom (S^m (V[1]), S^n (V[-1]))$ and
$\chi_W \in \prod_{m,n \ge 1} \Hom (S^m (W[1]), S^n (W[-1]))$ of the
Maurer-Cartan equation $\{\chi, \chi\} = 0$ defining the
$L_\infty$-bialgebra structures:
\[
\Phi \circ \chi_V = \chi_W \circ \Phi.
\]
Here, as in the definition of an $\Li$-morphism of $\Li$-algebroids,
$\chi_V$ and $\chi_W$ are understood as operators. For example,
$\chi_V \in \prod_{m,n \ge 1} \hbar^{n-1} \Hom (S^m (V[1]), S^n
(V[-1])) \cong \prod_{m,n \ge 1} \Hom (S^m (V[1]), S^n (V[1]))[-2]$.

C.~Bai, Y.~Sheng, and C.~Zhu \cite{bai-sheng-zhu} have considered an
interesting truncated version of $\Li$-bialgebras, namely 2-\emph{Lie
  bialgebras}, as well as morphisms of them. Since their notion of a
2-Lie bialgebra is not a straighforward truncation of the notion of an
$\Li$-bialgebra, it is not clear how our theory compares to
theirs. This could be an exciting topic of further study, especially
taking into account that Bai, Sheng, and Zhu present several examples
of 2-Lie bialgebras.
\end{example}

\subsection*{Acknowledgments}
The authors are grateful to Jim Stasheff for numerous useful remarks
and encouragement and to Theodore Voronov for stimulating discussions
and for drawing our attention to the works \cite{ThV4,ThV3}, where
similar ideas are developed.  The work of the second author was
supported by World Premier International Research Center Initiative
(WPI), MEXT, Japan, and a Collaboration grant from the Simons
Foundation (\#282349).

\bibliographystyle{amsalpha}
\bibliography{bi-lie}

\end{document}